\documentclass[12pt, reqno]{amsart}
\usepackage{amsmath, amsthm, amscd, amsfonts, amssymb, graphicx, color}
\usepackage[bookmarksnumbered, colorlinks, plainpages]{hyperref}

\usepackage[all]{xy}
\newtheorem{theorem}{Theorem}[section]
\newtheorem{lemma}[theorem]{Lemma}
\newtheorem{proposition}[theorem]{Proposition}
\newtheorem{corollary}[theorem]{Corollary}
\theoremstyle{definition}
\newtheorem{definition}[theorem]{Definition}
\newtheorem{example}[theorem]{Example}

\newtheorem{question}[theorem]{Question}
\theoremstyle{remark}
\newtheorem{remark}[theorem]{Remark}
\numberwithin{equation}{section}

\begin{document}
	\setcounter{page}{1}
	\title[\textit{DW}--$L$-Weakly Compact Operators]{\textit{DW}--$L$-Weakly Compact Operators Between Banach Lattices}
	
    \author[J.X. Chen]
       {Jin Xi Chen}
       \address{School of Mathematics, Southwest Jiaotong University, Chengdu 610031, China}
       \email{jinxichen@swjtu.edu.cn}

	 \author[Q. Wang]
	  {Qiaozhi Wang}
	 \address { School of Mathematics, Southwest Jiaotong University, Chengdu 610031, China}
\email{wangqiaozhi@my.swjtu.edu.cn}




	
\begin{abstract}
In this paper we introduce  a class of operators  which we call \textit{DW}--$L$-weakly compact  operators. They carry disjointly weakly compact subsets of a Banach lattice onto $L$-weakly compact sets. The class of \textit{DW}--$L$-weakly compact  operators is closely related to two recently introduced classes of operators: almost $L$-weakly compact operators and order $L$-weakly compact operators. We show that \textit{DW}--$L$-weakly compact  operators are  precisely those operators which are both almost $L$-weakly compact  and order $L$-weakly compact.  The adjoint and the  order structure of \textit{DW}--$L$-weakly compact operators are also considered.
\end{abstract}

\maketitle\baselineskip 5mm

\section{Introduction and preliminaries}
	
	 Throughout this paper, $X, Y$ will denote real Banach spaces, and $E, F$ will denote real Banach lattices. $B_X$ is the closed unit ball of $X$, and $E^+$ is the positive cone of $E$.  The solid hull of a subset $A$ of $E$ is denoted by $Sol(A):=\lbrace y\in E:|y| \leq |x|$ \ for \ some\ $x\in A\rbrace$. The order continuous part of $E$ is denoted by $E^a$.

\par Let us recall some notions and properties of $L$- and $M$-weakly compact operators introduced by P. Meyer-Nieberg \cite{Meyer1974}. A  norm bounded subset $A$ of a Banach lattice $E$ is said to be  \textit{$L$-weakly compact} if $\|x_n\|\rightarrow0$ for every disjoint sequence $(x_{n})$ contained in  $Sol(A)$ (see, e.g., \cite[Definition 3.6.1]{MeyerNieberg1991}). A bounded linear operator $T:X\to E$  is called an \textit{$L$-weakly compact operator} ($T\in \textmd{Lwc} (X,E)$) if $TB_X$ is  $L$-weakly compact in $E$. Every $L\,$-weakly compact set is relatively weakly compact, but the converse does not hold in general. In an $L$-space, $L\,$-weakly compact sets and relatively weakly compact sets coincide. Therefore, when $E$ is an $L$-space, $T:X\to E$ is $L$-weakly compact if and only if it is weakly compact. In duality with $L$-weakly compact operators are $M$-weakly compact operators.  An operator $T:E\to X$ is called an $M$-weakly compact operator ($T\in \textmd{Mwc}(E,X)$) if $\|Tx_n\|\to 0$ for every bounded disjoint sequence $(x_n)$ in $E$. Note that every $M$-weakly compact operator is automatically weakly compact. When $E$ is an $M$-space, $T:E\to X$ is $M$-weakly compact if and only if it is weakly compact. Also, the $L$- and  $M$-weakly compact operators are in duality between themselves. Precisely speaking, $T:X\to E$ is $L$-weakly compact if and only if its adjoint $T^{\prime}$ is $M$-weakly compact, and $T:E\to X$ is  $M$-weakly compact  if and only if its adjoint $T^{\prime}$ is $L$-weakly compact. For more details refer to \cite [Section 5.3]{Aliprantis2006} and \cite[Section 3.6]{MeyerNieberg1991}.

\par In the past decade, other classes of operators associated with $L$-weakly compact sets have  appeared successively. Bouras, Lhaimer and Moussa \cite{Bouras2018, LMB, Lhaimer2021} introduced  almost \textit{L-}(resp.  \textit{M-})weakly compact operators and order \textit{L-}(resp.  \textit{M-})weakly compact operators.

\begin{enumerate}

  \item  $T:X\to E$ is called \textit{ almost $L$-weakly compact} ($T\in \textmd{aLwc}(X, E)$) if $T$ carries relatively weakly compact subsets of $X$ onto $L$-weakly compact subsets of $E$.
  \item  $T:E\to X$  is called\textit{ almost $M$-weakly compact} ($T\in \textmd{aMwc}(E, X)$) if we have $f_{n}(Tx_n)\to 0$ for every bounded disjoint sequence $(x_n)$ in $E$ and every weakly convergent sequence $ (f_n)$ of $X^{\prime}$.
  \item  $T:E\to F$ is called \textit{order $L$-weakly compact} ($T\in \textmd{oLwc}(E,F)$) if $T$ carries order bounded subsets of $E$ onto  $L$-weakly compact  subsets of $F$.
  \item $T:E\to F$ is called \textit{order $M$-weakly compact} ($T\in \textmd{oMwc}(E,F)$) if $f_{n}(Tx_n)\to 0$ for every bounded disjoint sequence $(x_n)$ in $E$ and every order bounded sequence $(f_n)$ in $F^{\prime}$.
\end{enumerate}
 $T\in \textmd{aMwc}(E, X)$ if and only if $T^{\prime}\in \textmd{aLwc}(X^{\prime}, E^{\prime})$, and $T\in \textmd{oMwc}(E,F)$ if and only if $T^{\prime}\in \textmd{oLwc}(F^{\prime},E^{\prime})$. For further properties of aL(M)wc and oL(M)wc operators, refer to a series of papers \cite{Akay B, Alpay2024a, Bouras2018, Elbour,ElFahri2022, LMB, Lhaimer2021, Nouira}.

Recall that a bounded subset $A$ of $X$ is called a \textit{Dunford-Pettis }(resp. \textit{limited}) \textit{set} if every weakly null (resp. weak$^*$-null) sequence $(x^{\prime}_{n})$ of $X^\prime$ converges uniformly to zero on $A$ \cite{DPset, Bourgain1984}. Recently, Alpay, Emelyanov, Gorokhova \cite{Alpay2024b}  introduced two classes of operators which send Dunford-Pettis (resp. limited) sets onto $L$-weakly compact sets, that is, \textit{Dunford-Pettis} (resp. \textit{limitedly }) \textit{$L$-weakly compact operators} (\textit{abbr.} DP-Lwc, l-Lwc). Furthermore,  limitedly $L$-weakly compact operator was earlier considered by Oughajji and Moussa \cite{Ough} under the name of \textit{weak $L$-weakly compact operator}.

\par Observe that weakly compact sets,  order bounded sets, Dunford-Pettis sets and limited sets in Banach lattices belong to  a more general class of sets, the so-called disjointly weakly compact sets. Following W. Wnuk \cite{Wnuk1999}, we  call a  bounded subset $A$ of a Banach lattice $E$ a \textit{disjointly weakly compact set}   if $x_{n}\xrightarrow{w}0$ for every disjoint sequence $(x_n)$ of $Sol(A)$. Clearly, the disjointly weakly compact set is the weak version of the $L$-weakly compact set. The disjointly weak compactness in Banach lattices has been extensively studied by Xiang, Chen and Li \cite{disjointlyweaklycompact}.  Recently, the operators related to disjointly weakly compact sets were also defined. Chen and Feng \cite{ChenDW1, ChenDW2} considered the operators which carry disjointly weakly compact sets onto relatively compact (resp. Dunford-Pettis, limited) sets. That is, a bounded linear operator $T:E\to X$ is called \textit{$DW$-compact} (resp. \textit{$DW$-DP}, \textit{$DW$-limited}) if $T$ carries every disjointly weakly compact subset of $E$ onto a relatively compact (resp. Dunford-Pettis, limited) subset of $X$.

\par Inspired by the work of Bouras et al. \cite{Bouras2018, Lhaimer2021} and Alpay et al. \cite{Alpay2024a, Alpay2024b}, we introduce a new class of operators called\textit{ DW--}$L$-weakly compact (abbr. \textit{DW}-Lwc) operators, which carry disjointly weakly compact sets onto $L$-weakly compact sets. Clearly, every $DW$-Lwc operator is aLwc (resp. oLwc, DP-Lwc, l-Lwc) since weakly compact sets,  order bounded sets, Dunford-Pettis sets and limited sets in Banach lattices are all disjointly weakly compact \cite[Remark 2.4]{disjointlyweaklycompact}. In this paper, we show that $DW$-Lwc operators are indeed precisely those operators which are both aLwc and oLwc operators. Furthermore, we consider the adjoint of \textit{DW}-Lwc operators and the lattice-order structure of the space of \textit{DW}-Lwc operators.

\section{\textit{DW}--$L$-weakly compact operators on Banach lattices}

 Although the disjointly weak compactness is involved in  the lattice-order structure, it is preserved by continuity. The following lemma appeared in \cite[Corollary 0.2]{RdpsetandV*sets} and will be frequently used in the sequel.
  
  \begin{lemma}\label{old lemma}\rm\cite{RdpsetandV*sets}
  Let $T:E\to F$ be a continuous linear operator between two Banach lattices. Then $T(K)$ is likewise disjointly weakly compact for every disjointly weakly compact subset $K$ of $E$.
  \end{lemma}
  
\begin{definition}
A bounded linear operator $T:E\to F$ between Banach lattices is called a \textit{DW}--$L$-weakly compact (abbr. \textit{DW}-Lwc) operator if  $T(A)$ is $L$-weakly compact in $F$ for  every disjointly weakly compact subset $A$ of $E$.  The collection of all \textit{DW}-Lwc operators from $E$ into $F$ is denoted by $\textit{DW-}\textmd{Lwc}(E,F)$.
\end{definition}
Clearly, every \textit{DW}-comapct operator into $F$ is \textit{DW}-Lwc when $F$ has order continuous norm. On the other hand, if $F$ is a discrete Banach lattice with order continuous norm, then every \textit{DW}-Lwc operator into $F$ is \textit{DW}-compact. 
\par We can characterize \textit{DW}-Lwc operators in terms of $L$-weakly compact operators and disjointly weakly compact operators. Following \cite[Definition 3.1] {disjointlyweaklycompact}, an operator $T:X\to  E$ from a Banach space $X$ into a Banach lattice $E$ is called a \textit{disjointly weakly compact operator} if $T(B_X)$ is a disjointly weakly compact subset of $E$. 

\begin{theorem}\label{characterization of DW-Lwc operators}
		For a bounded linear operator $T:E\rightarrow F$ between Banach lattices the following statements are equivalent:

		\begin{enumerate}
			\item $T:E\rightarrow F$ is a \textit{DW}-Lwc operator.
			
			\item For an arbitrary Banach space $X$ and each disjointly weakly compact operator $S:X\rightarrow E$, $TS$ is an $L$-weakly compact operator.
			\item For each disjointly weakly compact operator $S:\ell^1\rightarrow E$, $TS$ is an $L$-weakly compact operator.
		\end{enumerate}
\end{theorem}

\begin{proof}
 It suffices  to prove (3)$\Rightarrow$(1). To this end, let $A$ be a disjointly weakly compact subset of $E$  and let $(f_n)$ be a bounded disjoint sequence in $F^{\prime }$. We have to show that $\sup_{x\in A}|f_{n}(Tx)|\to 0 \,\, (n\to \infty)$. Otherwise, passing to a subsequence if necessary, there would exist some $\epsilon_{0}>0$ and a sequence $(x_n)_{n=1}^{\infty}\subseteq A$ such that $|f_{n}(Tx_{n})|>\epsilon_{0}$. Let us define
  an operator $S:\ell^1\rightarrow E$ by
 $$S\lambda=\sum\limits_{n=1}^{\infty} \lambda_{n}x_{n}, \,\,\, \lambda=(\lambda_n)\in \ell^1 $$Then, $S$ is a disjointly weakly compact operator since $S(B_{\ell^1})\subseteq \overline{\{x_{n}:n\in \mathbb{N}\}^{bc}}$, where $\overline{\{x_{n}:n\in \mathbb{N}\}^{bc}}$, the closed absolutely convex hull of $(x_n)$, is disjointly weakly compact (see \cite[Theorem 2.3]{disjointlyweaklycompact}). Therefore, by our hypothesis,   $TS:\ell^1\rightarrow F$ is an $L$-weakly compact operator. Hence, it follows that
  $$0<\epsilon_{0}<|f_{n}(Tx_{n})|\leq |\langle f_{n}, TSe_{n}\rangle|\leq\sup_{\lambda\in B_{\ell^1}}|\langle f_{n}, TS\lambda\rangle|\to 0 \,\,\, (n\to \infty)$$Here, $(e_n)$ is the standard unit vector basis of $\ell^1$. This leads to a contradiction, which implies that $T$ is \textit{DW}-Lwc.
\end{proof}

We know that every $L$-weakly compact operator between Banach lattices is \textit{DW}-Lwc. The following result tells us when the converse holds.
\begin{theorem} \label{DW-Lwc is Lwc}
	For a  Banach lattice $E$ the following statements are equivalent:
	\begin{enumerate}
        \item $E^{\prime}$ has order continuous norm.
		\item Every \textit{DW}-Lwc operator from $E$ into an arbitrary Banach lattice  is an $L$-weakly compact operator.
		\item Every \textit{DW}-Lwc operator $T:E\to c_0$ is an $L$-weakly compact operator.
    \end{enumerate}
\end{theorem}
	
\begin{proof}
\par (1)$\Rightarrow$(2) This follows from the fact that $B_E$ is disjointly weakly compact if and only if  $E^{\prime}$ has order continuous norm.
 \par (2)$\Rightarrow$(3) Obvious.
 \par (3)$\Rightarrow$(1) To prove that $E^{\prime}$ has order continuous norm we have to show that every order bounded disjoint sequence in $E^{\prime}$ converges in norm to zero. To this end, let $(x^{\prime}_n)\subset E^{\prime}$ be an order bounded disjoint sequence. Since $x^{\prime}_{n}\xrightarrow{w}0$, we can define an operator $T:E\to c_0$ as $Tx=(x^{\prime}_{n}(x))$ for $x\in E$. We claim that $T$ is \textit{DW}-Lwc. Indeed, let $A$ be an arbitrary disjointly weakly compact subset of $E$. From \cite[Theorem 2.3]{disjointlyweaklycompact} it follows that $\sup_{x\in A}|x^{\prime}_{n}(x)|\to 0$. This implies that $T(A)$ is relatively norm compact and hence $L$-weakly compact in $c_0$. That is, $T$ is a \textit{DW}-Lwc operator. By our hypothesis, $T$ is $L$-weakly compact. Therefore, $T(B_E)$ is $L$-weakly compact and hence relatively compact in $c_0$. It follows that $\|x^{\prime}_{n}\|\to 0\,\,\, (n\to \infty)$.
 \end{proof}
 
 Let us recall that a subset $A$ of  $E$ is said to be an \textit{approximately order bounded set} if, for every $\varepsilon > 0$, there exists $u\in E^+$ such that $A\subseteq[-u, u] + \varepsilon B_E$. Also, a bounded linear operator $T:X\to E$  is called \emph{semi-compact} if $T(B_X)$ is approximately order bounded in $E$. We know that every $L$-weakly compact set is approximately order bounded  \cite[Proposition 3.6.2]{MeyerNieberg1991}, and hence every $L$-weakly compact operator is semi-compact. On the other hand, if $E$ has order continuous norm, then every semi-compact operator $T:X\to E$ is $L$-weakly compact.

  \begin{example}
  
  \begin{enumerate}
  \item  Consider the natural inclusion operator $J_{\ell_{1}, c_0}:\ell_1 \to c_0$. Since every disjointly weakly compact subset of $\ell_1$ is relatively compact \cite[Theorem 2.6]{ChenDW1}, $J_{\ell_{1}, c_0}$ is \textit{DW}-Lwc. However, $J_{\ell_{1}, c_0}$ is not semi-compact.
                    
  \item  Consider the natural inclusion operator $J_{\ell_{p},\ell_\infty } : \ell_p \to \ell_\infty$ \,\, $(1<p<\infty)$. Clearly, $J_{\ell_{p},\ell_\infty }$  is semi-compact. However, $J_{\ell_{p},\ell_\infty }$ is not \textit{DW}-Lwc. Indeed, $B_{\ell_{p}}$ is disjointly weakly compact, and $(J_{\ell_{p},\ell_\infty }(e_n))_{n=1}^{\infty}$ is a disjoint sequence in the solid hull of $J_{\ell_{p}, \ell_\infty }(B_{\ell_{p}})$, but $(J_{\ell_{p}, \ell_\infty }(e_n))_{n=1}^{\infty}$ does not converge in norm to zero in $\ell_\infty$.
\end{enumerate}
\end{example} 
 
  Now we give some conditions under which a \textit{DW}-Lwc (resp. semi-compact) operator is semi-compact (resp. \textit{DW}-Lwc). As an immediate consequence of  Theorem \ref{DW-Lwc is Lwc} we have the following corollary.
   
\begin{corollary}\label{thm:domain-semicompact}
For a Banach lattice $E$ the following statements are equivalent:
	\begin{enumerate}
		\item $E^{\prime}$ has order continuous norm.
		\item For every Banach lattice $F$, every \textit{DW}-Lwc operator $T:E\to F$ is semi-compact.
		\item Every \textit{DW}-Lwc operator $T:E\to c_0$ is semi-compact.
	\end{enumerate}
\end{corollary}

\begin{theorem}\label{thm:codomain-semicompact}
	Let $E$ be a Banach lattice. Then $E$ has order continuous norm if and only if, for every Banach lattice $G$, every semi-compact operator $T:G\to E$ is a \textit{DW-}Lwc operator.
\end{theorem}

\begin{proof}
	If $E$ has order continuous norm, then every semi-compact operator $T:G\to E$ is $L$-weakly compact, and hence it is \textit{DW-}Lwc.

\par For the converse, it suffices to show that $[-u,u]$ is a weakly compact subset of $E$ for every $u\in E^+$. Denote by $E_u$  the principal ideal generated by $u$ in $E$, equipped with the order unit norm. We know that $E_u$ is an \textit{AM}-space with a unit.  $B_{E_u}=[-u,u]$ is disjointly weakly compact  in $E_u$. Now, consider the natural inclusion  $T : E_{u} \to E$. Obviously, $T$ is semi-compact, and hence $T$ is, by our hypothesis, a \textit{DW}-Lwc operator. Therefore,  $T(B_{E_u}) = [-u, u]$ is an $L$-weakly compact set in $E$, and hence $[-u,u]$ is weakly compact. That is, the norm on $E$ is order continuous.
\end{proof} 
 
   Hence, from Theorem \ref{DW-Lwc is Lwc}, Corollary \ref{thm:domain-semicompact} and Theorem \ref{thm:codomain-semicompact}, we can easily see that if both $E^{\prime}$ and $F$ have order continuous norms, then the classes of \textit{DW}-Lwc operators, semi-compact operators, and $L$-weakly compact operators coincide.

\par Recall that a Banach lattice $E$ has the\textit{ positive Schur property} if $0\leq x_{n}\xrightarrow{w}0$ in $E$ implies $\|x_n\|\to0$, equivalently if every relatively weakly compact subset of $E$ is $L$-weakly compact \cite[Corollary 3.6.8]{MeyerNieberg1991}. In \cite[Corollary 2.7]{disjointlyweaklycompact} it was proved that $E$ has the positive Schur property if and only if every disjointly weakly compact subset of $E$ is $L$-weakly compact. Therefore, a Banach lattice $E$ has the positive Schur property if and only if the identity operator $I$ on $E$ is \textit{DW}-Lwc. Generally, we have the following result on the positive Schur property. 

\begin{theorem}
A Banach lattice $E$ has the positive Schur property if and only if every bounded linear operator from an arbitrary Banach lattice $G$ into $E$ is \textit{DW}-Lwc.
\end{theorem}

\begin{proof} Assume first that $E$ has the positive Schur property. Let $T:G\to E$ be a bounded linear operator and let $A$ be a disjointly weakly compact subset of $G$. Then, from Lemma \ref{old lemma} it follows that $T(A)$ is likewise disjointly weakly compact in $E$. Therefore, $T(A)$ is $L$-weakly compact in $E$ since $E$ has the positive Schur property. This implies that $T$ is \textit{DW}-Lwc.

\par For the converse, consider the identity operator $I:E\to E$.
\end{proof}

It is obvious that every $DW$-Lwc operator is aLwc and oLwc. The following result shows that $DW$-Lwc operators are indeed precisely those operators which are both aLwc and oLwc. That is, $\textit{DW-}\textmd{Lwc}(E, F)=\textmd{aLwc}(E, F)\cap\textmd{oLwc}(E, F)$.
	
\begin{theorem}\label{thm:main_intersection} For a bounded linear operator $T:E\to F$ between Banach lattices the following assertions are equivalent:
\begin{enumerate}
  \item  $T$ is a $DW$-Lwc operator.
  \item  $T(Sol(A))$ is an $L$-weakly compact subset of $F$ for every relatively weakly compact subset $A$ of $E$.
  \item  $T$ is an  almost $L$-weakly compact (aLwc) and order $L$-weakly compact (oLwc) operator.
\end{enumerate}

\end{theorem}
	
\begin{proof} (1)$\Rightarrow$(2)  This is obvious because every relatively weakly compact subset of a Banach lattice is disjointly weakly compact \cite[Theorem 4.34]{Aliprantis2006}.

\par (2)$\Rightarrow$(3)  We can easily see that $T$ is an oLwc operator since every order bounded set is contained in the solid hull of a singleton.

\par (3)$\Rightarrow$(1)  To this end,  let $A$ be a disjointly weakly compact subset of $E$, and let $(g_n)$ be a bounded disjoint sequence in $F^{\prime}$. By definition, we can assume without loss of generality that $A$ is solid. We have to show that $\sup_{x\in A}|g_{n}(Tx)|=\sup_{x\in A}|\langle |T^{\prime}g_{n}|, x\rangle|\to0$ \, ($n\to\infty$). Otherwise, passing to a subsequence if necessary, there would exist some $\epsilon_{0}>0$ and a sequence $(x_n)_{n=1}^{\infty}\subseteq A^+$ such that $|T^{\prime}g_{n}|x_{n}>\epsilon_{0}>0$. Clearly, $|T^{\prime}g_n|\xrightarrow{w^*}0$ in $E^{\prime}$ since $T$ is oLwc. Hence, as we did in \cite[Lemma 2.2]{Xiang2022} we can ﬁnd a strictly increasing subsequence $(n_m)$
of $\mathbb{N}$ and a disjoint sequence $(y_m)\subset A^+$ satisfying

$$0\leq y_{m}\leq x_{n_{m+1}}\,\, \textmd{and}\,\,\,\, |T^{\prime}g_{n_{m+1}}|y_{m}\geq \epsilon_{0} \,\textmd{for}\, m \,\textmd{sufficiently \,large}$$Applying the Riesz-Kantorovich formula, we can pick a sequence $(z_m)\subset A$ such that

 $$|z_{m}|\leq y_m\,\, \textmd{and} \,\,|(T^{\prime}g_{n_{m+1}})z_{m}|=|g_{n_{m+1}}(Tz_{m})|\geq \epsilon_{0} \,\,\,\textmd{when}\,\, m \,\, \textmd{is}\,\,\textmd{sufficiently \,large}$$Clearly,  $(z_m)$ is a disjoint sequence from $A$, and hence $z_{m}\xrightarrow{w}0$ since $A$ is disjointly weakly compact. Therefore, it follows that $g_{n_{m+1}}(Tz_{m})\to 0$\,\, ($m\to\infty$) because, by our hypothesis, $T$ is aLwc. This leads to a contradiction.
\end{proof}
\begin{remark}
In view of the equivalence of (1) and (2) in Theorem \ref{thm:main_intersection}, one can ask\textit{ whether or not a disjointly weakly compact subset of $E$ is necessarily contained in the solid hull of a relatively weakly compact set}. Here, the authors are not able to answer it.
\end{remark}

We let  $\textit{DW-}\textmd{Lwc}_+(E, F)$, $\textmd{aLwc}_+(E, F)$ and $\textmd{oLwc}_+(E, F)$ denote the collection of all positive $DW$-Lwc operators, positive aLwc operators and positive oLwc operators from $E$ into $F$, respectively.

\begin{corollary}\label{positive DW-Lwc}
    $\textit{DW-}\textmd{Lwc}_+(E, F) = \textmd{aLwc}_+(E, F)$.
\end{corollary}
	
\begin{proof}
	By Theorem \ref{thm:main_intersection}, we have $\textit{DW-}\textmd{Lwc}_{+}(E, F) = \textmd{oLwc}_{+}(E, F) \cap \textmd{aLwc}_{+}(E, F)$. Recently, Alpay and Gorokhova \cite[Proposition 3.2]{Alpay2024a} showed that  $\textmd{aLwc}_+(E, F) \subseteq \textmd{oLwc}_+(E, F)$ holds. Consequently, $\textit{DW-}\textmd{Lwc}_+(E, F) = \textmd{aLwc}_+(E, F)$.
\end{proof}

\begin{remark}
El Fahri, Khabaoui and H'michane \cite{ElFahri2022} introduced a class of operators which they called  order almost $L$-weakly compact operators. An operator $T:E\to F$ between Banach lattices is called \textit{order almost $L$-weakly compact} if $T$ sends each order bounded weakly compact set onto an $L$-weakly compact set. Let $\textmd{o-aLwc}(E,F)$ (resp. $\textmd{l-Lwc}(E,F) $ ) denote the collection of all order almost $L$-weakly compact (resp. limitedly $L$-weakly compact) operators from $E$ into $F$. Then we have the following implications:
$$\textmd{aLwc}(E, F)\cup\textmd{oLwc}(E, F)\subseteq \textmd{o-aLwc}(E,F)\subseteq \textmd{l-Lwc}(E,F)$$since every singleton in $E$ is clearly an order bounded weakly compact set. If $0\leq T\in \textmd{o-aLwc}(E,F)$, then, by \cite[Proposition 3.2]{Alpay2024a}, we have $T\in\textmd{oLwc}(E,F)$. That is,
$\textmd{o-aLwc}_+(E, F) = \textmd{oLwc}_+(E, F)$.
\end{remark}

It should be noted that a Banach lattice $E$ has order continuous norm if and only if every order interval in $E$ is weakly compact. Hence we have $\textmd{aLwc}(E,F)\subseteq\textmd{oLwc}(E,F)$ if $E$ has order continuous norm. Clearly, the following is an immediate consequence of this observation.
\begin{corollary}
   Let $E$ and  $F$ be Banach lattices such that $E$ has order continuous norm. Then $\textit{DW-}\textmd{Lwc}(E, F) = \textmd{aLwc}(E, F)$.
\end{corollary}

The following  example shows that an oLwc (resp. aLwc) operator is not necessarily \textit{DW}-Lwc.
\begin{example}
	
\begin{enumerate}
  \item  Let $I:\ell_{p}\to\ell_{p}$\,\, ($1<p<\infty$) be the identity operator on $\ell_p$. Since every order interval in $\ell_p$ is an $L$-weakly compact set, $I$ is oLwc. On the other hand, $B_{\ell_{p}}$ is a disjointly weakly compact set, but it is not an $L$-weakly compact set. Therefore, $I$ is not \textit{DW}-Lwc.
  \item  It should be noted that every bounded linear operator from $L^{\infty}[0,1]$ into $c_0$ is a Dunford-Pettis operator since $L^{\infty}[0,1]$ is a Grothendieck space and has the Dunford-Pettis property. Therefore,  the  operator $T:L^{\infty}[0,1]\to{c_0}$ defined by $$Tf=\left(\int_{0}^{1}f(t)r_{1}(t)\textrm{d}t, \int_{0}^{1}f(t)r_{2}(t)\textrm{d}t, \cdot\cdot\cdot, \int_{0}^{1}f(t)r_{n}(t)\textrm{d}t, \cdot\cdot\cdot\right),$$where  $(r_n)_{n=1}^{\infty}$ is the sequence of Rademacher functions on $[0,1]$, is  Dunford-Pettis, and hence aLwc  since relatively compact sets and $L$-weakly compact sets coincide in $c_0$.  Clearly, $B_{L^{\infty}[0,1]}=[-\mathbf{1},\textbf{1}]$ is a disjointly weakly compact set. However, $TB_{L^{\infty}[0,1]}$ is not relatively compact in $c_0$ (and hence not $L$-weakly compact) since 
      $$\sup_{f\in [-\mathbf{1},\textbf{1}]}\left |\int_{0}^{1}f(t)r_{n}(t)\textrm{d}t \right  |\geq\int_{0}^{1}r_{n}(t) r_{n}(t)\textrm{d}t =\int_{0}^{1}\textrm{d}t=1\nrightarrow0\,\, \, (n\to\infty)$$This implies that $T$ is not \textit{DW}-Lwc.
  \end{enumerate}
  \end{example}
		
\section{The adjoint operators of \textit{DW}-Lwc operators}

Let us recall that $T:X\to E$ is $L$-weakly compact if and only if its adjoint $T^{\prime}$ is $M$-weakly compact. Also, $T:E\to X$ is  $M$-weakly compact  if and only if its adjoint $T^{\prime}$ is $L$-weakly compact. Similarly, $T\in \textmd{aMwc}(E, X)$ if and only if $T^{\prime}\in \textmd{aLwc}(X^{\prime}, E^{\prime})$, and $T\in\textmd{oMwc}(E,F)$ if and only if $T^{\prime}\in \textmd{oLwc}(F^{\prime},E^{\prime})$. Naturally, we are interested in the duality for \textit{DW}-Lwc operators.

\begin{definition}\label{def:DW-Mwc_general} Let $E$ and $F$ be Banach lattices.
\begin{enumerate}
  \item  An operator $T:E\to F$  is called a \textit{DW-}Mwc operator ($T \in \textit{DW-}\textmd{Mwc}(E, F)$) if for every bounded disjoint sequence $(x_n)$ in $E$ and every disjointly weakly compact set $B \subseteq F^{\prime }$, we have $\lim_{n} \sup_{f \in B} |f(Tx_n)| = 0$.
  \item  An operator $T:E\to F^{\prime}$ is called a \textit{DW-}Mwc$^*$ operator  ($T \in \textit{DW-}\textmd{Mwc}^{*}(E, F^{\prime})$) if for every bounded disjoint sequence $(x_n)$ in $E$ and every disjointly weakly compact set $A\subseteq F$, we have $\lim_{n} \sup_{y \in A} |\langle y,Tx_n\rangle| = 0$.
  \end{enumerate}

\end{definition}

It is obvious that $\textmd{Mwc}(E,F^{\prime})\subseteq \textit{DW-}\textmd{Mwc}^{*}(E, F^{\prime})$  for two Banach lattices $E$ and $F$. For the reverse direction we have the following result.

\begin{theorem}\label{thm:3.3}
	Let $F$ be a Banach lattice. $F^{\prime}$ has  order continuous norm if and only if $\textit{DW-}\textmd{Mwc}^*(E, F^{\prime}) = \textmd{Mwc}(E, F^{\prime})$ for every Banach lattice $E$.

\end{theorem}

\begin{proof}
	The necessity follows from a classical result that   $F^{\prime}$ has order continuous norm if and only if $B_F$ is disjointly weakly compact (see, e.g., \cite[Theorem 2.4.14]{MeyerNieberg1991}).

     \par For the sufficiency, assume by way of contradiction that $F^\prime $ does not have  order continuous norm. Then there would exist an order bounded disjoint sequence $(f_n)\subseteq(F^{\prime})^+$  such that  $\|f_n\|= 1$ for all $n$. Thus, we define a bounded linear operator $T: \ell_1\to F^\prime$ by $Tx=\sum_{n=1}^\infty \lambda_n f_n$ for every $x=(\lambda_n)\in \ell_1$.

\par First, we have $T \notin \textmd{Mwc}(\ell_1, F^\prime)$. Indeed, let $(e_n)$ be the standard unit vector basis of $\ell_1$. It is a disjoint sequence from $B_{\ell_1}$. However, $\|Te_{n}\| = \|f_n\| = 1 \nrightarrow0$.

 \par Now, we claim that $T\in\textit{DW-}\textmd{Mwc}^*(\ell_1, F^\prime)$. To this end, let $(x_k)_{k=1}^{\infty}$ be a disjoint sequence from $B_{\ell_1}$, where $x_{k}=(\lambda_{1}^{(k)}, \lambda_{2}^{(k)}, \cdot\cdot\cdot,\lambda_{n}^{(k)},\cdot\cdot\cdot )\in \ell_1$. Therefore, from Corollary 2.4.3 of \cite{MeyerNieberg1991} it follows that $|x_{k}|\xrightarrow{w^*}0$ in $\ell_1$.  Let $A\subseteq F$ be an arbitrary disjointly weakly compact set, and let  $c_{n} := \sup_{y\in A} |f_{n}(y)|$ \,\, $(n=1, 2, \cdot\cdot\cdot)$. In view of \cite[Theorem 2.3]{disjointlyweaklycompact}, we have $c_{n}\to 0$ \,($n\to\infty$), i.e., $\alpha=(c_n)\in c_0$.  For every $y\in A$, we have
 
 $$|\langle y, Tx_k\rangle|=\left|\sum_{n=1}^{\infty} \lambda_{n}^{(k)} f_{n}(y)\right|\leq\sum_{n=1}^{\infty}|\lambda_{n}^{(k)}||f_{n}(y)|\leq\sum_{n=1}^{\infty}c_{n}|\lambda_{n}^{(k)}|$$Taking the supremum over $A$, we obtain
 
 $$\sup_{y\in A}|\langle y, Tx_k\rangle|\leq\sum_{n=1}^{\infty}c_{n}|\lambda_{n}^{(k)}|=\langle \alpha, |x_k|\rangle\to0\,\,\, (k\to\infty)$$since $\alpha\in c_0$ and $|x_{k}|\xrightarrow{w^*}0$ in $\ell_1$. By definition, $T\in\textit{DW-}\textmd{Mwc}^*(\ell_1, F^\prime)$.
 \end{proof}

It should be noted that, by Lemma \ref{old lemma}, disjointly weakly compact subsets of $F$ are likewise disjointly weakly compact in $F^{\prime\prime}$. Therefore, the following set inclusions hold:

	$$\textmd{Mwc}(E, F^\prime ) \subseteq \textit{DW-}\textmd{Mwc}(E, F^\prime ) \subseteq \textit{DW-}\textmd{Mwc}^*(E, F^\prime).$$As a consequence of Theorem \ref{thm:3.3}, we have $\textit{DW-}\textmd{Mwc}(E, F^\prime )=\textit{DW-}\textmd{Mwc}^*(E, F^\prime)$ when $F^\prime$ has order continuous norm.

\par As for $\textit{DW-}\textmd{Mwc}(E, F)$, we know that $\textmd{Mwc}(E, F)=\textit{DW-}\textmd{Mwc}(E, F)$ when $F^{\prime\prime}$ has order continuous norm.

\begin{example}
Let $E=\ell_p$ \,\,$(1\leq p<\infty)$ and $F=c_0$. Clearly, $F^{\prime\prime}=\ell_\infty$ does not have order continuous norm. Consider the natural inclusion operator $J_{\ell_{p}, c_0}:\ell_p \to c_0$. Let $(e_n)_{n=1}^{\infty}$ be the standard unit vector basis of $\ell_p$. Then  $(e_n)_{n=1}^{\infty}$ is a disjoint sequence from $B_{\ell_p}$, and $\|J_{\ell_{p}, c_0}(e_n)\|=1\nrightarrow0$. Hence, $J_{\ell_{p}, c_0}\notin\textmd{Mwc}(\ell_{p}, c_0)$. On the other hand, we claim that $J_{\ell_{p}, c_0}\in\textit{DW-}\textmd{Mwc}(\ell_{p}, c_0)$. To this end, let $(x_n)_{n=1}^{\infty}\subset B_{\ell_p}$ be a disjoint sequence and let $B\subset \ell_{1}=(c_0)^{\prime}$ be a disjointly weakly compact set. Then $(J_{\ell_{p}, c_0}(x_n))$ is likewise a disjoint sequence in $c_0$, and $B$ is an $L$-weakly compact subset of $\ell_1$ since $\ell_1$ has the Schur property \cite[Corollary 2.7]{disjointlyweaklycompact}. Therefore, $(J_{\ell_{p}, c_0}(x_n))_{n=1}^{\infty}$ converges uniformly to zero on $B$.
\end{example}
	
There are two questions which the authors are not able to answer and have to leave  still open.
\begin{question}
(1)  Is the norm of $F^{\prime\prime}$ order continuous norm if $\textit{DW-}\textmd{Mwc}(E, F)=\textmd{Mwc}(E, F)$ for every Banach lattice $E$?

(2) Does $F^{\prime}$ have order continuous norm if $\textit{DW-}\textmd{Mwc}(E, F^\prime )=\textit{DW-}\textmd{Mwc}^*(E, F^\prime)$ for every Banach lattice $E$?

\end{question}

Next, we turn our attention to the adjoint operators of \textit{DW}-Lwc and \textit{DW}-Mwc operators.
\begin{proposition}\label{adjoint}
	For a bounded linear operator $T: E \to F$ between Banach lattices  the following assertions hold:
	\begin{enumerate}
        \item $T\in\textit{DW-}\textmd{Lwc}(E, F)\iff T^{\prime}\in\textit{DW-}\textmd{Mwc}^*(F^{\prime }, E^\prime)$.
		\item $T\in\textit{DW-}\textmd{Mwc}(E, F)\iff T^{\prime}\in\textit{DW-}\textmd{Lwc}(F^{\prime}, E^{\prime})$.
		\item $T^{\prime}\in\textit{DW-}\textmd{Mwc}(F^{\prime}, E^{\prime}) \implies T\in\textit{DW-}\textmd{Lwc}(E, F)$.
        \item $\textit{DW-}\textmd{Mwc}(E, F) = oMwc(E, F)\cap\textmd{aMwc}(E, F)$.
        \item $\textit{DW-}\textmd{Mwc}_{+}(E, F)=\textmd{aMwc}_{+}(E, F)$.
	\end{enumerate}
\end{proposition}

\begin{proof} The proofs of these implications are based on the equivalent characterizations of $L$-weakly compact sets in a Banach lattice and its dual space \cite[Proposition 3.6.2 \& 3.6.3]{MeyerNieberg1991}.

\begin{equation*}
    	\begin{aligned}(1)
 \quad T \in &\textit{DW-}\textmd{Lwc}(E, F)\\
            &\iff T(A)\text{ is an $L$-weakly compact set in $F$ for every disjointly weakly } \\
            & \qquad\qquad\textmd{ compact set $A\subset E$}\\
            &\iff\sup_{x\in A}|f_{n}(Tx)|\to 0\text{ for every disjointly weakly compact set $A\subseteq E$}\\
            &\qquad \qquad \text{  and every bounded disjoint sequence } (f_{n})\subseteq F^{\prime}\\
		    &\iff\sup_{x\in A}|(T^{\prime}f_{n})x|\to 0 \text{ for every bounded disjoint sequence $(f_{n})\subseteq F^\prime$} \\
            & \qquad\qquad \text{ and  every disjointly weakly compact set $A\subseteq E$} \\
		    &\iff T^{\prime} \in \textit{DW-}\textmd{Mwc}^*(F^{\prime }, E^\prime).
	    \end{aligned}   
	\end{equation*}
\begin{equation*}
    \begin{aligned}(2)
\quad  T \in &\textit{DW-}\textmd{Mwc}(E, F)\\
            &\iff\sup_{f\in B}|f(Tx_n)|\to 0 \text{ for every bounded disjoint sequence  $(x_{n})\subseteq E$ and}\\
            &\qquad\qquad  \text{  every disjointly weakly compact set $B\subseteq F^{\prime}$} \\
            &\iff\sup_{f\in B}|(T^{\prime}f)x_n|\to 0 \text{ for every disjointly weakly compact set $B\subseteq F^\prime$ } \\
            & \qquad \qquad \text{ and every bounded disjoint sequence $(x_{n})\subseteq E$} \\
            &\iff T^{\prime}(B)\text{ is an $L$-weakly compact set in $E^\prime$ for every disjointly weakly } \\
            & \qquad\qquad\textmd{ compact set  $B\subseteq F^\prime$}\\
		    &\iff T^{\prime} \in \textit{DW-}\textmd{Lwc}(F^{\prime }, E^{\prime}).
	    \end{aligned}   
	\end{equation*} 
\begin{equation*}
    \begin{aligned}(3)
\quad T^{\prime}\in\textit{ DW-}\textmd{Mwc}(F^{\prime}, E^{\prime})\implies T^{\prime}\in\textit{DW-}\textmd{Mwc}^*(F^{\prime}, E^{\prime})\implies T\in\textit{DW-}\textmd{Lwc}(E, F).
\end{aligned}   
\end{equation*} 
\par (4) The proof of  $\textit{DW-}\textmd{Mwc}(E, F)\subset\textmd{oMwc}(E, F) \cap \textmd{aMwc}(E, F)$ is straightforward since  relatively weakly compact sets and order bounded sets in $F^\prime$ are all disjointly weakly compact. For the reverse inclusion, let $T\in\textmd{oMwc}(E, F) \cap \textmd{aMwc}(E, F)$. Then $T^{\prime}\in \textmd{oLwc}(F^{\prime}, E^{\prime}) \cap \textmd{aLwc}(F^{\prime}, E^{\prime})$. Therefore, from Theorem \ref{thm:main_intersection} it follows that $T^{\prime}\in \textit{DW-}\textmd{Lwc}(F^{\prime}, E^{\prime})$ and hence, by statement (2), $T\in\textit{DW-}\textmd{Mwc}(E, F)$.

\par(5) If $T\in\textmd{aMwc}_{+}(E, F)$, then $T^{\prime}\in\textmd{aLwc}_{+}(F', E^\prime)$. From Corollary \ref{positive DW-Lwc} it follows that $T^{\prime}\in \textit{DW-}\textmd{Lwc}_{+}(F^{\prime}, E^\prime)$. By statement (2), we have $T\in\textit{DW-}\textmd{Mwc}_{+}(E, F)$.
\end{proof}

\begin{remark}\label{I is DWMwc}For a Banach lattice $E$, let $I_E$ (resp. $I_{E^\prime}$) denote the identity operator on $E$ (resp. $E^\prime$).
\begin{enumerate}
  \item  By Proposition \ref{adjoint}(2) we have $I_{E}\in \textit{DW-}\textmd{Mwc}(E) \iff I_{E^\prime} \in \textit{DW-}\textmd{Lwc}( E^{\prime})$ $\iff$ $E^{\prime}$  has the positive Schur property.
  \item  $I_{E^\prime}\in\textit{DW-}\textmd{Mwc}^*(E^\prime) \iff I_{E}\in\textit{DW-}\textmd{Lwc}(E)$ $\iff$ $E$  has the positive Schur property.
\end{enumerate}

\end{remark}

Now we are in a position to consider  the algebraic and lattice-order structures of the class of $\textit{DW-}\textmd{Lwc}$ (resp. $\textit{DW-}\textmd{Mwc}$, $\textit{DW-}\textmd{Mwc}^*$) operators.

For Banach lattices $E$ and $F$,  it is routine to verify that  $\textit{DW-}\textmd{Lwc}(E, F)$  (resp. $\textit{DW-}\textmd{Mwc}(E, F)$, $\textit{DW-}\textmd{Mwc}^*(E, F^\prime)$), equipped with the operator norm, is a Banach space.

\begin{proposition}\label{ideal}
For Banach lattices $E$, $F$ and $G$ the following assertions hold:
\begin{enumerate}
  \item   If $T\in\textit{DW-}Lwc(E, F)$ and $S\in \mathcal{L}(G,E)$, then $TS\in\textit{DW-}\textmd{Lwc}(G, F)$. 
  \item  If $T\in\textit{DW-}\textmd{Lwc}(E, F)$ and $Q\in \textmd{aLwc}(F,G)$ (resp. $Q\in \textmd{oLwc}(F,G)$), then $QT\in\textit{DW-}\textmd{Lwc}(E, G)$. 
  \item  $\textit{DW-}\textmd{Lwc}(E)$ is a closed right ideal of $ \mathcal{L}(E)$.
  \item  If $T\in\textit{DW-}\textmd{Mwc}(E, F)$ and $S\in \mathcal{L}(F,G)$, then $ST\in\textit{DW-}\textmd{Mwc}(E, G)$.
  \item  $\textit{DW-}\textmd{Mwc}(E)$ is a closed left ideal of $ \mathcal{L}(E)$.
\end{enumerate}
\end{proposition}

\begin{proof}
It suffices to prove statement (1), (2) and (4).
\par (1) Let $T\in\textit{DW-}\textmd{Lwc}(E, F)$ and $S\in \mathcal{L}(G,E)$. Let $A$ be a disjointly weakly compact subset of $G$. By Lemma \ref{old lemma}, $S(A)$ is likewise disjointly weakly compact in $E$. Therefore, $TS(A)$ is an $L$-weakly compact subset of $F$. This implies that $TS\in\textit{DW-}\textmd{Lwc}(G, F)$.
\par (2) Let $T\in\textit{DW-}\textmd{Lwc}(E, F)$ and $Q\in\textmd{oLwc}(F,G)$. Let $A\subseteq E$ be a disjointly weakly compact set. For each $\epsilon>0$, there exists $0\leq y\in F^a$ such that $T(A)\subseteq [-y,y]+\epsilon B_F$. Since $Q[-y,y]$ is $L$-weakly compact, there exists $0\leq u\in G^a$ satisfying 
$$QT(A)\subseteq Q[-y,y]+\epsilon Q(B_F)\subseteq [-u,u]+\epsilon B_{G}+\epsilon \|Q\|B_G$$which implies that $QT(A)$ is $L$-weakly compact, and hence $QT$ is $DW$-Lwc.
\par (4) If $T\in\textit{DW-}\textmd{Mwc}(E, F)$ and $S\in \mathcal{L}(F,G)$, then by assertion (1) and Proposition \ref{adjoint}(2) it follows that $T^{\prime}S^{\prime}\in\textit{DW-}\textmd{Lwc}$. This implies that $ST\in\textit{DW-}\textmd{Mwc}(E, G)$.
\end{proof}

 As usual, $\mathcal{L}^{r}(E,F)$ denotes the space of regular operators between Banach lattices. $\textit{DW-}\textmd{Lwc}_{+}(E, F)$ (resp. $\textit{DW-}\textmd{Mwc}_{+}(E, F)$, $\textit{DW-}\textmd{Mwc}^{*}_{+}(E, F^\prime)$) denotes the collection of positive \textit{DW-}Lwc (resp. $\textit{DW-}\textmd{Mwc}$, $\textit{DW-}\textmd{Mwc}^{*}$) operators. $\mathrm{Span}(\textit{DW-}\textmd{Lwc}_+(E, F))$ is the linear subspace generated by $\textit{DW-}\textmd{Lwc}_+(E, F))$ in $\mathcal{L}(E,F)$, actually in $\mathcal{L}^{r}(E,F)$. For a regular operator $T$ between Banach lattices, let $\|T\|_r$ denote the regular norm of $T$. It is well-known that $(\mathcal{L}^{r}(E,F), \|\cdot\|_r)$ is a Banach space, and furthermore $(\mathcal{L}^{r}(E,F), \|\cdot\|_r)$  is a Dedekind complete Banach lattice if the target space $F$, in addition, is Dedekind complete \cite[Proposition 1.3.6]{MeyerNieberg1991}.

\begin{proposition}\label{dw-lwc domination}
Let $S, T\in \mathcal{L}(E, F)$  satisfy  $\pm S \leq T$. If $T$ is $ \textit{DW--}\textmd{Lwc}$ \rm(resp.  $\textit{DW-}\textmd{Mwc}$\rm), then $S$ is likewise $\textit{DW-}\textmd{Lwc}$ (resp.  $\textit{DW-}\textmd{Mwc}$). Furthermore, we have
 $S\in\mathrm{Span}(\textit{DW-}\textmd{Lwc}_{+}(E, F))$ (resp. $\mathrm{Span}(\textit{DW-}\textmd{Mwc}_{+}(E, F))$).
\end{proposition}

\begin{proof}

Assume first that $T$ is $ \textit{DW-}\textmd{Lwc}$.  Let $A$ be an arbitrary disjointly weakly compact subset of $E$. By definition, we may assume without loss of generality that $A$ is solid. For all $x \in A$, we have $|Sx|\leq T(|x|) \in T(A)$, which implies that $S(A)\subseteq Sol(T(A))$. Since $T \in \textit{DW-}\textmd{Lwc}(E, F)$, $Sol(T(A))$ is an $L$-weakly compact set. Therefore, $S(A)$ is $L$-weakly compact, and hence $S$ is a \textit{DW-}Lwc operator. Now, $S=\frac{1}{2}(T+S)-\frac{1}{2}(T-S)$ with $T+S, T-S\in \textit{DW-}\textmd{Lwc}_{+}(E, F)$. Now, it is clear that $S\in\mathrm{Span}(\textit{DW-}\textmd{Lwc}_{+}(E, F))$.

\par On the other hand, if $T$ is a $\textit{DW-}\textmd{Mwc}$ operator, then $\pm S^{\prime}\leq T^{\prime}\in \textit{DW-}\textmd{Lwc}(F^{\prime},E^{\prime})$. Therefore,  $S^{\prime}\in \textit{DW-}\textmd{Lwc}(F^{\prime},E^{\prime})$, which implies that  $S\in\textit{DW-}\textmd{Mwc}(E,F)$. Similarly, $S\in\mathrm{Span}(\textit{DW-}\textmd{Mwc}_{+}(E, F))$.
\end{proof}

\begin{remark}
By Corollary \ref{positive DW-Lwc} and Proposition \ref{adjoint}(5), we know that
 $$\mathrm{Span}(\textit{DW-}\textmd{Lwc}_{+}(E, F))=\mathrm{Span}(\textmd{aLwc}_{+}(E, F)),$$  
 $$\mathrm{Span}(\textit{DW-}\textmd{Mwc}_{+}(E, F))=\mathrm{Span}(\textmd{aMwc}_{+}(E, F)).$$By \cite[Theorem 6 \& 7]{Akay B},   $\mathrm{Span}(\textit{DW-}\textmd{Lwc}_+(E,F))$ (resp. $\mathrm{Span}(\textit{DW-}\textmd{Mwc}_+(E, F))$ if $F$ is Dedekind complete),  equipped with the regular norm, is a  Dedekind complete Banach lattice.
\end{remark}

As for the lattice properties of \textit{DW}-Mwc$^*$ operators, we have the following result.
\begin{theorem}\label{Domination Property}
	Let $E$ and $F$ be two Banach lattices. 
	\begin{enumerate}
       \item Let $T, S \in \mathcal{L}(E, F^{\prime})$ satisfy $\pm S \leq T$. If $T \in \textit{DW-}\textmd{Mwc}^*(E, F^{\prime})$, then $S \in \mathrm{Span}(\textit{DW-}\textmd{Mwc}^*_+(E, F^{\prime}))$.

        \item $\mathrm{Span}(\textit{DW-}\textmd{Mwc}^*_+(E, F^{\prime}))$ is a $\|\cdot\|_r$-closed order ideal in $(\mathcal{L}^r(E, F^{\prime }), \|\cdot\|_r)$. Therefore, $\mathrm{Span}(\textit{DW-}\textmd{Mwc}^*_+(E, F^{\prime}))$, equipped with the regular norm, is a Dedekind complete Banach lattice in its own right.

	\end{enumerate}
\end{theorem}

\begin{proof}
	
(1)  Assume that $T \in \textit{DW-}\textmd{Mwc}^*(E, F^{\prime})$. Let $(x_n)$ be a bounded disjoint sequence in $E$ and let $A$ be a  disjointly weakly compact solid subset of $F$. From the inequalities
 $$\sup_{y\in A}|\langle y,Sx_n\rangle|\leq\sup_{y\in A}\langle |y|,|Sx_n|\rangle\leq\sup_{y\in A}\langle |y|,T|x_n|\rangle=\sup_{u\in A}|\langle u,T|x_n|\rangle|\to0\,\,(n\to\infty)$$it follows that $S\in \textit{DW-}\textmd{Mwc}^*(E, F^{\prime})$. Now, since $S=\frac{1}{2}(T+S)-\frac{1}{2}(T-S)$ with $T+S$, $T-S\in \textit{DW-}\textmd{Mwc}^{*}_{+}(E, F^{\prime})$,  we have $S\in\mathrm{Span}(\textit{DW-}\textmd{Mwc}^{*}_{+}(E, F^{\prime}))$.

 \par (2) It is obvious that  $\mathrm{Span}(\textit{DW-}\textmd{Mwc}^{*}_{+}(E, F^{\prime}))$ is a linear subspace  of $\mathcal{L}^{r}(E, F^\prime)$. To prove that  $\mathrm{Span}(\textit{DW-}\textmd{Mwc}^{*}_{+}(E, F^{\prime}))$ is an order ideal, let $S, T\in\mathcal{L}^{r}(E, F^{\prime})$ be such that $|S|\leq|T|$ and $T\in\mathrm{Span}(\textit{DW-}\textmd{Mwc}^{*}_{+}(E, F^{\prime}))$. Since $T\in\mathrm{Span}(\textit{DW-}\textmd{Mwc}^{*}_{+}(E, F^{\prime}))$, there exist $T_{1}, T_{2}\in\textit{DW-}\textmd{Mwc}^{*}_{+}(E, F^{\prime})$ satisfying $T=T_{1}-T_{2}$. Hence, $\pm T\leq T_{1}+T_{2}$. It follows that $|T|\leq T_{1}+T_{2}\in \textit{DW-}\textmd{Mwc}^{*}_{+}(E, F^{\prime})$. By statement $(1)$, we have $|T|\in\textit{DW-}\textmd{Mwc}^{*}_{+}(E, F^{\prime})$. Since $\pm S\leq|S|\leq|T|$, once again by statement $(1)$, we have  $S\in\mathrm{Span}(\textit{DW-}\textmd{Mwc}^{*}_{+}(E, F^{\prime}))$, which implies that $\mathrm{Span}(\textit{DW-}\textmd{Mwc}^{*}_{+}(E, F^{\prime}))$ is an order ideal in $\mathcal{L}^{r}(E, F^{\prime})$.
  
 \par To prove that $\mathrm{Span}(\textit{DW-}\textmd{Mwc}^{*}_{+}(E, F^{\prime}))$  is $\|\cdot\|_r$-closed in $(\mathcal{L}^{r}(E, F^{\prime}), \|\cdot\|_r)$, let $(T_{k})_{k=1}^{\infty}\subseteq\mathrm{Span}(\textit{DW-}\textmd{Mwc}^{*}_{+}(E, F^{\prime}))$ and $T\in\mathcal{L}^{r}(E, F^{\prime})$ be such that $T_{k}\xrightarrow{\|\cdot\|_r}T$. From the inequality $| |T_k|-|T| |\leq|T_{k}-T|$, it follows
 
 $$\||T_k|-|T|\|\leq\||T_k|-|T|\|_{r}\leq\|T_k-T\|_{r}\to0\,\,\, (k\to\infty).$$Since each $|T_k|\in \textit{DW-}\textmd{Mwc}^{*}_{+}(E, F^{\prime})$ and $\textit{DW-}\textmd{Mwc}^{*}(E, F^{\prime})$ is $\|\cdot\|$-closed in $\mathcal{L}(E, F^{\prime})$, we have $|T|\in\textit{DW-}\textmd{Mwc}^{*}_{+}(E, F^{\prime})$. Therefore, $T\in\mathrm{Span}(\textit{DW-}\textmd{Mwc}^{*}_{+}(E, F^{\prime}))$. That is, 
 $\mathrm{Span}(\textit{DW-}\textmd{Mwc}^{*}_{+}(E, F^{\prime}))$  is $\|\cdot\|_r$-closed in $(\mathcal{L}^{r}(E, F^{\prime}), \|\cdot\|_r)$.

\end{proof}
 

\vskip 6mm


\noindent \textbf{Funding} This work was partly supported by the  Natural Science Foundation of China (Grant
No.12571143).\\




\end{document}